\documentclass[reqno,11pt]{amsart}
\usepackage{amsmath, latexsym, amsfonts, amssymb, amsthm, amscd}
\usepackage{graphics,epsf,psfrag}

\numberwithin{equation}{section}

\newtheorem{theorem}{Theorem}[section]
\newtheorem{lemma}[theorem]{Lemma}
\newtheorem{proposition}[theorem]{Proposition}

\newcommand{\ind}{\mathbf{1}}

\newcommand{\R}{\mathbb{R}}
\newcommand{\Z}{\mathbb{Z}}
\newcommand{\N}{\mathbb{N}}
\renewcommand{\tilde}{\widetilde}
\renewcommand{\hat}{\widehat}

\newcommand{\bP}{{\ensuremath{\mathbf P}} }
\newcommand{\bE}{{\ensuremath{\mathbf E}} }

\DeclareMathSymbol{\leqslant}{\mathalpha}{AMSa}{"36} % nicer `smaller or equal'
\DeclareMathSymbol{\geqslant}{\mathalpha}{AMSa}{"3E} % nicer `larger or equal'
\DeclareMathSymbol{\eset}{\mathalpha}{AMSb}{"3F}     % nicer `emptyset'
\newcommand{\dd}{\,\text{\rm d}}             % a straight d for differentials
\newcommand{\bbE}{{\ensuremath{\mathbb E}} }

\newcommand{\bbP}{{\ensuremath{\mathbb P}} }

\newcommand{\ga}{\alpha}
\newcommand{\gb}{\beta}
\newcommand{\gga}{\gamma}            % \gg already exists...

\newcommand{\gep}{\varepsilon}       % \ge already exists...

\newcommand{\gr}{\rho}

\newcommand{\gD}{\Delta}

\newcommand{\go}{\omega}

\newcommand{\gl}{\lambda}

\newcommand{\gs}{\sigma}

\makeatletter
\def\captionfont@{\footnotesize}
\def\captionheadfont@{\scshape}

\long\def\@makecaption#1#2{%
  \vspace{2mm}
  \setbox\@tempboxa\vbox{\color@setgroup
    \advance\hsize-6pc\noindent
    \captionfont@\captionheadfont@#1\@xp\@ifnotempty\@xp
        {\@cdr#2\@nil}{.\captionfont@\upshape\enspace#2}%
    \unskip\kern-6pc\par
    \global\setbox\@ne\lastbox\color@endgroup}%
  \ifhbox\@ne % the normal case
    \setbox\@ne\hbox{\unhbox\@ne\unskip\unskip\unpenalty\unkern}%
  \fi
  \ifdim\wd\@tempboxa=\z@ % this means caption will fit on one line
    \setbox\@ne\hbox to\columnwidth{\hss\kern-6pc\box\@ne\hss}%
  \else % tempboxa contained more than one line
    \setbox\@ne\vbox{\unvbox\@tempboxa\parskip\z@skip
        \noindent\unhbox\@ne\advance\hsize-6pc\par}%
\fi
  \ifnum\@tempcnta<64 % if the float IS a figure...
    \addvspace\abovecaptionskip
    \moveright 3pc\box\@ne
  \else % if the float IS NOT a figure...
    \moveright 3pc\box\@ne
    \nobreak
    \vskip\belowcaptionskip
  \fi
\relax
}
\makeatother
\def\writefig#1 #2 #3 {\rlap{\kern #1 truecm
\raise #2 truecm \hbox{#3}}}

\newcommand{\tf}{\textsc{f}}
\newcommand{\M}{\textsc{M}}

\begin{document}
\title[Disorder relevance for pinning models]{
Fractional moment bounds  and disorder relevance
\\
for pinning models}

\author{Bernard Derrida}
\address{Laboratoire de Physique Statistique, 
D\'epartement de Physique, 
\'Ecole Normale Sup\'erieure 
24, rue Lhomond, 
75231 Paris Cedex 05,  France }
\email{derrida\@@lps.ens.fr }
\author{Giambattista Giacomin}
\address{
  Universit{\'e} Paris  Diderot (Paris 7) and Laboratoire de Probabilit{\'e}s et Mod\`eles Al\'eatoires (CNRS U.M.R. 7599),
U.F.R.                Math\'ematiques, Case 7012,
                2 place Jussieu, 75251 Paris cedex 05, France
}
\email{giacomin\@@math.jussieu.fr}
\author{Hubert Lacoin}
\address{
  Universit{\'e} Paris  Diderot (Paris 7) and Laboratoire de Probabilit{\'e}s et Mod\`eles Al\'eatoires (CNRS U.M.R. 7599),
U.F.R.                Math\'ematiques, Case 7012,
                2 place Jussieu, 75251 Paris cedex 05, France
}
\email{hlacoin\@@gmail.com}
\author{Fabio Lucio Toninelli}
\address{
Laboratoire de Physique, ENS Lyon (CNRS U.M.R. 5672), 46 All\'ee d'Italie, 
69364 Lyon cedex 07, France
}
\email{fabio-lucio.toninelli@ens-lyon.fr}
\date{\today}

\begin{abstract}
  We study the critical point of directed pinning/wetting models with
  quenched disorder.  The distribution $K(\cdot)$ of the location of
  the first contact of the (free) polymer with the defect line is
  assumed to be of the form $K(n)=n^{-\ga-1}L(n)$, with $\ga\ge0$ and
  $L(\cdot)$ slowly varying. The model undergoes a (de)-localization
  phase transition: the free energy (per unit length) is zero in the
  delocalized phase and positive in the localized phase.  For
  $\ga<1/2$ disorder is irrelevant: quenched and
  annealed critical points coincide for small disorder, as well as
  quenched and annealed critical exponents \cite{cf:Ken,cf:T_cmp}.
  The same has been proven also for $\ga=1/2$, but under the
  assumption that $L(\cdot)$ diverges sufficiently fast at infinity,
  a hypothesis that is not satisfied in the $(1+1)$-dimensional
  wetting model considered in \cite{cf:FLNO,cf:DHV}, where $L(\cdot)$
  is asymptotically constant.  Here we prove that, if $1/2< \ga<1$ or
  $\ga >1$, then quenched and annealed critical points differ whenever
  disorder is present, and we give the scaling form of their
  difference for small disorder.  In agreement with the so-called
  Harris criterion, disorder is therefore relevant in this case. In
  the marginal case $\ga=1/2$, under the assumption that $L(\cdot)$
  vanishes sufficiently fast at infinity, we prove that the difference
  between quenched and annealed critical points, which is 
  smaller than any power of the disorder strength, is positive:
  disorder is {\sl marginally relevant}. Again, the case considered in
  \cite{cf:FLNO,cf:DHV} is out of our analysis and remains open.

  The results are achieved by setting the parameters of the model so
  that the annealed system is localized, but close to criticality, and
  by first considering a quenched system of size that does not exceed the
  correlation length of the annealed model.  In such a regime we can
  show that the expectation of the partition function raised to a
  suitably chosen power $\gamma\in (0,1)$ is small.  We then exploit
  such an information to prove that the expectation of the same
  fractional power of the partition function goes to zero with the
  size of the system, a fact that immediately entails that the
  quenched system is delocalized.
\end{abstract}

\maketitle

\section{Introduction}

Pinning/wetting models with quenched disorder describe the random
interaction between a directed polymer and a one-dimensional {\sl
defect line}. In absence of interaction, a typical polymer
configuration is given by $\{(n,S_n)\}_{n\ge0}$, where
$\{S_n\}_{n\ge0}$ is a Markov Chain on some state space $\Sigma$ (for
instance, $\Sigma=\Z^d$ for $(1+d)$-dimensional directed polymers),
and the initial condition $S_0$ is some fixed element of $\Sigma$
which by convention we call $0$. The defect line, on the other hand,
is just $\{(n,0)\}_{n\ge0}$. The polymer-line interaction is
introduced as follows: each time $S_n=0$ ({\sl i.e.}, the polymer
touches the line at step $n$) the polymer gets an energy
reward/penalty $\epsilon_n$, which can be either positive or
negative. In the situation we consider here, the $\epsilon_n$'s are
independent and identically distributed (IID) random variables, with
positive or negative mean $h$ and variance $\beta^2\ge0$.

Up to now, we have made no assumption on the Markov Chain. The
physically most interesting case is the one where the distribution
$K(\cdot)$ of the first return time, call it $\tau_1$, of $S_n$ to $0$
has a power-law tail: $K(n):=\bP(\tau_1=n)\approx n^{-\alpha-1}$, with
$\alpha\ge0$. This framework allows to cover various situations
motivated by (bio)-physics: for instance, $(1+1)$-dimensional wetting
models \cite{cf:FLNO,cf:DHV} ($\ga=1/2$; in this case $S_n\ge0$, and
the line represents an impenetrable wall), pinning of
$(1+d)$-dimensional directed polymers on a columnar defect ($\ga=1/2$
if $d=1$ and $\ga=d/2-1$ if $d\ge2$), and the Poland-Scheraga model of
DNA denaturation (here, $\ga\simeq 1.15$ \cite{cf:dna}).  This is a
very active field of research, and not only from the point of view of
mathematical physics, see {\sl. e.g.}
\cite{cf:coluzzi} and references therein.  We refer to \cite[Ch. 1]{cf:GB} and references
therein for further discussion.

The model undergoes a localization/delocalization phase transition:
for any given value $\gb$ of the disorder strength, if the average
pinning intensity $h$ exceeds some critical value $h_c(\gb)$ then the
polymer typically stays tightly close to the defect line and the free
energy is positive.  On the contrary, for $h< h_c(\gb)$ the free
energy vanishes and the polymer has only few contacts with the defect:
entropic effects prevail. The annealed model, obtained by averaging the
Boltzmann weight with respect to disorder, is exactly solvable, and
near its critical point $h_c^{ann}(\gb)$ one finds that the annealed
free energy vanishes like $(h-h_c^{ann}(\gb))^{\max(1,1/\ga)}$
\cite{cf:Fisher}. In particular, the annealed phase transition is
first order for $\ga>1$ and second order for $\ga<1$, and it gets
smoother and smoother as $\ga$ approaches $0$.

A very natural and intriguing question is whether and how randomness
affects critical properties. The scenario suggested by the 
{\sl Harris criterion} \cite{cf:Harris} is the following: disorder
should be irrelevant for $\ga<1/2$, meaning that quenched critical
point and critical exponents should coincide with the annealed ones if
$\beta$ is small enough, and relevant for $\ga>1/2$: they should
differ for every $\gb>0$.  In the marginal case $\ga=1/2$, the Harris
criterion gives no prediction and there is no general consensus on
what to expect: renormalization-group considerations led Forgacs {\sl
  et al.}  \cite{cf:FLNO} to predict that disorder is irrelevant (see also the recent 
   \cite{cf:GN}),
while Derrida {\sl et al.} \cite{cf:DHV} concluded for marginal
relevance: quenched and annealed critical points should differ for
every $\gb>0$, even if the difference is zero at every perturbative
order in $\gb$.

\medskip

The mathematical understanding of these questions witnessed
remarkable progress recently, and we summarize here the state of the
art (prior to the the present contribution). %%G
\begin{enumerate}
\item A lot is now known on the  {\sl irrelevant-disorder regime}. 
In particular, it was proven in \cite{cf:Ken} (see
  \cite{cf:T_cmp} for an alternative proof) that quenched and annealed
  critical points and critical exponents coincide for $\beta$ small
  enough. Moreover, in \cite{cf:GT_irrel} a small-disorder expansion
  of the free energy, worked out in \cite{cf:FLNO}, was rigorously
  justified.

\item In the {\sl strong-disorder regime}, for which the Harris
  criterion makes no prediction, a few results were obtained recently.
  In particular, in \cite{cf:T_AAP} it was proven that for any given
  $\ga>0$ and, say, for Gaussian randomness, $h_c(\gb)\ne
  h_c^{ann}(\gb)$ for $\beta$ large enough, and the asymptotic
  behavior of $h_c(\gb)$ for $\gb\to\infty$ was computed. These
  results were obtained through upper bounds on fractional moments of
  the partition function.  Let us mention by the way that the
  fractional moment method allowed also to compute exactly
  \cite{cf:T_AAP} the quenched critical point of a {\sl diluted
    wetting model} (a model with a built-in strong-disorder
  limit); the same result was obtained in \cite{cf:BCT} via a rigorous
  implementation of renormalization-group ideas. Fractional moment
  methods  have proven to be useful also for other classes of disordered models
\cite{cf:AM,cf:ASFH,cf:BPP,cf:ED}. %%G

\item The {\sl relevant-disorder regime} is only partly understood. In
\cite{cf:GT_cmp} it was proven that the free-energy critical exponent
differs from the quenched one whenever $\gb>0$ and 
$\alpha>1/2$. %However, proving that there is a critical point shift has appeared considerably harder.
However, the arguments in \cite{cf:GT_cmp} do not imply the critical point shift.
%%G
Nonetheless, 
the critical point shift issue  has been recently solved for a {\sl hierarchical
version} of the model, introduced in \cite{cf:DHV}. The hierarchical
model also depends on the parameter $\alpha$, and in \cite{cf:GLT} it
was shown that $h_c(\gb)-h_c^{ann}(\gb)\approx \beta^{2\ga/(2\ga-1)}$
for $\gb$ small (upper and lower bounds of the same order are
proven).

\item In the {\sl marginal case $\ga=1/2$} it was proven in
  \cite{cf:Ken,cf:T_cmp} that the difference $h_c(\gb)-h_c^{ann}(\gb)$
  vanishes faster than any power of $\gb$, for $\gb\to0$.  Before
  discussing lower bounds on this difference, one has to be more
  precise on the tail behavior of $K(n)$, the probability that the
  first return to zero of the Markov Chain $\{S_n\}_n$ occurs at $n$:
  if $K(n)=n^{-(1+1/2)}L(n)$ with $L(\cdot)$ slowly varying (say, a
  logarithm raised to a positive or negative power), then the two
  critical points coincide for $\beta$ small \cite{cf:Ken,cf:T_cmp} if
  $L(\cdot)$ diverges sufficiently fast at infinity so that
  \begin{equation}
    \label{eq:condL}
\sum_{n=1}^\infty\frac1{n L(n)^2}\, <\,\infty.    
  \end{equation}
  The case of
  the $(1+1)$-dimensional wetting model \cite{cf:DHV} corresponds
  however to the case where $L(\cdot)$ behaves like a constant at
  infinity, and the result just mentioned does not apply.

The case $\ga=1/2$ is open also for the hierarchical model mentioned
above.

\end{enumerate}

\medskip

In the present work we prove that if $\ga\in(1/2,1)$ or $\ga>1$ then
quenched and annealed critical points differ for every $\beta>0$, and
$h_c(\gb)-h_c^{ann}(\gb)\approx \gb^{2\ga/(2\ga-1)}$ for
$\gb\searrow0$ ({\sl cf.} Theorem \ref{th:a121} for a more precise
statement). 
%For the case $\ga=1$, see the {\sl note added in proof} at the end of this paper. %%G
In the case $\ga=1/2$, while we do not
prove that $h_c(\gb)\ne h_c^{ann}(\gb)$ in all cases in which condition
\eqref{eq:condL} fails, we do prove such a result if the function
$L(\cdot)$ vanishes sufficiently fast at infinity.  Of course,
$h_c(\gb)- h_c^{ann}(\gb)$ turns out to be exponentially small for
$\gb\searrow0$.

We wish to emphasize that, although the Harris criterion is expected to 
be applicable to a large variety of disordered models, rigorous results
are very rare: let us mention however
\cite{cf:CCFS,cf:vD}.% in the context of ferromagnetic spin models and  percolation.

\medskip

Starting from next section, we will forget the full Markov structure 
of the polymer, and retain only the fact that the set of points of contact
with the defect line, $\tau:=\{n\ge 0:S_n=0\}$, is a renewal process
under the law $\bP$ of the Markov Chain.

\section{Model and main results}

Let $\tau:=\{\tau_0,\tau_1,\ldots\}$ be a renewal sequence started
from $\tau_0=0$ and with inter-arrival law $K(\cdot)$, {\sl i.e.,}
$\{\tau_i-\tau_{i-1}\}_{i\in \N:=\{1, 2, \ldots\}}$ are IID integer-valued random
variables with law $\bP(\tau_1=n)=K(n)$ for every $n\in\N$.   We assume that
$\sum_{n\in\N}K(n)=1$ (the renewal is recurrent) and that there exists
$\ga>0$ such that
\begin{equation}
  \label{eq:K}
  K(n)= \frac{L(n)}{n^{1+\alpha}}
\end{equation}
with $L(\cdot)$ a function that varies slowly at infinity, {\sl i.e.}, $L: (0, \infty) \to (0, \infty)$ is 
 measurable  and
such that $L(rx)/L(x)\to1$ when $x\to\infty$, for every $r>0$.  We
refer to \cite{cf:RegVar} for an extended treatment of slowly varying
functions, recalling just that examples of $L(x)$ include $(\log (1+x))^b$,
any $b \in \R$, and any (positive, measurable) function admitting
a positive limit at infinity (in this case we say that $L(\cdot)$ is {\sl trivial}).
Dwelling a bit more on nomenclature, $x \mapsto x^\gr L(x)$
is a {\sl regularly varying function of exponent } $\gr$,
so $K(\cdot)$ is just the restriction to the natural numbers of
a regularly varying function of exponent $-(1+\ga)$.

\medskip

We let $\gb\ge0$, $h\in\R$ and
$\go:=\{\go_n\}_{n\ge1}$ be a sequence of IID centered random variables
with unit variance and finite exponential moments. The law of $\go$ is 
denoted by $\bbP$ and the corresponding expectation  by $\bbE$.

For $a,b\in\{0,1,\ldots\}$ with $a\le b$ we let $Z_{a,b,\omega}$ be the
partition function for the system on the interval
$\{a,a+1,\ldots,b\}$, with zero boundary conditions at both endpoints:
\begin{equation}
\label{eq:Z}
Z_{a,b,\omega}=\bE\left(\left.e^{\sum_{n=a+1}^b(\beta\go_n+h)
\ind_{\{n\in\tau\}}}\ind_{\{b\in\tau\}}\right|a\in\tau\right),
\end{equation}
where $\bE$ denotes expectation with respect to the law $\bP$ of the
renewal.  One may rewrite $Z_{a,b,\omega}$ more explicitly as
\begin{equation}
  \label{eq:Z+}
  Z_{a,b,\omega}=\sum_{\ell=1}^{b-a}\sum_{i_0=a<i_1<\ldots<i_\ell=b}
\prod_{j=1}^\ell K(i_j-i_{j-1})e^{h\ell+\beta\sum_{j=1}^\ell \omega_{i_j}},
\end{equation}
with the convention that $Z_{a,a,\omega}=1$. 
Notice that, when writing  $n \in \tau$, we are interpreting $\tau$ as a subset
of $\N \cup \{0\}$ rather than as a sequence of random variables.
 We will write for
simplicity $Z_{N,\go}$ for $Z_{0,N,\go}$ (and in that case the
conditioning on $0\in\tau$ in \eqref{eq:Z} is superfluous since
$\tau_0=0$).  In absence of disorder ($\gb=0$), it is convenient to
use the notation
\begin{equation}
\label{eq:Zpure}
  Z_N(h):=\bE\left(e^{h\sum_{n=1}^N\ind_{\{n\in\tau\}}}
\ind_{\{N\in\tau\}}\right)=\bE\left(e^{h|\tau\cap\{1,\ldots,N\}|}\ind_{\{N\in\tau\}}
\right), 
\end{equation}
for the partition function.

We mention that the recurrence assumption $\sum_{n\in\N}K(n)=1$
entails no loss of generality, since one can always reduce to this
situation via a redefinition of $h$ ({\sl cf.} \cite[Ch.~1]{cf:GB}).

\medskip
As usual the {\sl quenched free energy} is defined as
\begin{equation}
\label{eq:F}
  \tf(\beta,h)=\lim_{N\to\infty}\frac1N \log Z_{N,\go}.
\end{equation}
It is well known ({\sl cf.} for instance \cite[Ch. 4]{cf:GB}) that the limit
\eqref{eq:F} exists $\bbP(\dd\go)$-almost surely and in $\mathbb
L^1(\bbP)$, and that it is almost-surely independent of $\go$.
Another well-established fact is that $\tf(\beta,h)\ge0$, which
immediately follows from $Z_{N,\go}\ge K(N)\exp(\gb \go_N+h)$.  This
allows to define, for a given $\gb\ge0$, the critical point $h_c(\gb)$
as
\begin{equation}
  \label{eq:hc}
  h_c(\gb):=\sup\{h\in\R:\;\tf(\beta,h)=0\}.
\end{equation}
It is well known that $h>h_c(\gb)$ corresponds to the {\sl localized
  phase} where typically $\tau$ occupies a non-zero fraction of
$\{1,\ldots,N\}$ while, for $h<h_c(\gb)$, $\tau\cap\{1,\ldots,N\}$
contains with large probability at most $O(\log N)$ points
\cite{cf:GTdeloc}. We refer to \cite[Ch.s 7 and 8]{cf:GB} for further
literature and discussion on this point.

\medskip

In analogy with the quenched free energy, the {\sl annealed free
  energy} is defined by
\begin{equation}
\label{eq:Fann}
  \tf^{ann}(\gb,h)\, :=\, \lim_{N\to\infty}\frac1N\log \bbE Z_{N,\go}=
\tf(0,h+\log\M(\gb)),
\end{equation}
with 
\begin{equation}
  \label{eq:M}
  \M(\gb):= \bbE(e^{\beta\go_1}).
\end{equation}
We see therefore that the annealed free energy is just the free energy
of the pure model ($\beta=0$) with a different value of $h$.  The pure
model is exactly solvable \cite{cf:Fisher}, and we collect here a few
facts we will need in the course of the paper.

\begin{theorem}\cite[Th. 2.1]{cf:GB} 
\label{th:pure} For the pure model $h_c(0)=0$.
  Moreover, there exists a slowly varying function $\hat L(\cdot)$
  such that for $h>0$ one has
  \begin{equation}
    \label{eq:Fpure}
    \tf(0,h)=h^{1/\min(1,\ga)}\hat L(1/h).
  \end{equation}
In particular,
\begin{enumerate}
\item if $\bE(\tau_1)=\sum_{n\in\N}n\,K(n)<\infty$ (for instance, if
  $\ga>1$) then $\hat L(1/h)\stackrel{h\searrow0}\sim 1/\bE(\tau_1)$.
\item if $\ga\in(0,1)$, then $\hat L(1/h)=C_\ga h^{-1/\ga}R_\ga(h)$ where
$C_\ga$ is an explicit constant and $R_\ga(\cdot)$ is the 
function, unique up to asymptotic equivalence,  that satisfies
$R_\ga ( b^\ga L(1/b)) \stackrel{b \searrow 0}\sim b$.
\end{enumerate}
\end{theorem}
As a consequence of Theorem \ref{th:pure} and \eqref{eq:Fann}, the
annealed critical point is simply given by
\begin{equation}
\label{eq:hann}
  h^{ann}_c(\beta):=\sup\{h:\tf^{ann}(\gb,h)=0\}=-\log \M(\beta).
\end{equation}

Via Jensen's inequality one has immediately that $\tf(\gb,h)\le
\tf^{ann}(\gb, h)$ and as a consequence $h_c(\gb)\ge h_c^{ann}(\gb)$,
and the point of the present paper is to understand when this last
inequality is strict.  In this respect, let us recall that the
following is known: if $\ga\in(0,1/2)$, then $h_c(\gb)=
h_c^{ann}(\gb)$ for $\beta$ small enough \cite{cf:Ken,cf:T_cmp}. Also
for $\ga =1/2$ it has been shown that $h_c(\gb)= h_c^{ann}(\gb)$ if
$L(\cdot) $ diverges sufficiently fast (see below).  Moreover,
assuming that $\bbP(\go_1>t)>0$ for every $t>0$, one has that for
every $\ga>0$ and $L(\cdot)$ there exists $\beta_0<\infty$ such that
$h_c(\gb)\ne h_c^{ann}(\gb)$ for $\gb>\beta_0$ \cite{cf:T_AAP}:
quenched and annealed critical points differ for {\sl strong
  disorder}. The strategy we develop here addresses the complementary
situations: $\ga>1/2$ and {\sl small disorder} (and also the case
$\ga=1/2$ as we shall see below).

\bigskip

Our first result concerns the case $\ga>1$:
\medskip

\begin{theorem}
\label{th:a>1}
  Let $\alpha>1$. There exists $a>0$ such that for every $\beta\le1$
  \begin{equation}
    \label{eq:a>1}
h_c(\beta)-h_c^{ann}(\beta)\ge a \beta^2.
  \end{equation}
Moreover, $h_c(\gb)>h_c^{ann}(\beta)$ for every $\gb>0$.
\end{theorem}
\medskip

Since $h_c(\gb)\le h_c(0)=0$ 
and $h_c^{ann}(\gb)\stackrel{\gb\searrow0}\sim-\gb^2/2$, we conclude
that the inequality \eqref{eq:a>1} is, in a sense, of the optimal
order in $\gb$. 
Note that $h_c(\gb)\le h_c(0)$ is just a consequence of Jensen's inequality:
\begin{eqnarray}
Z_{N,\go}&= &Z_N(h)\frac{\bE \left(e^{\sum_{n=1}^N(\gb\go_n+h)\ind_{\{n\in\tau\}}}\ind_{\{N\in\tau\}}\right)}
{\bE\left(e^{h\sum_{n=1}^N\ind_{\{n\in\tau\}}}\ind_{\{N\in\tau\}}\right)}\\\nonumber
&\ge& Z_N(h)\exp\left[
\gb\sum_{n=1}^N\go_n \frac{\bE\left(\ind_{\{n\in\tau\}}e^{h|\tau\cap\{1,\ldots,N\}|}\ind_{\{N\in\tau\}}\right)}
{\bE\left(e^{h|\tau\cap\{1,\ldots,N\}|}\ind_{\{N\in\tau\}}\right)}
\right],
\end{eqnarray}
from which $\tf(\gb,h)\ge\tf(0,h)$ and therefore $h_c(\gb)\le h_c(0)$ immediately follows from $\bbE(\go_n)=0$.
This can be made sharper in the sense that from the
explicit bound in \cite[Th.~5.2(1)]{cf:GB}  one directly extract also
that $h_c(\gb) \le -b \gb^2$ for a suitable $b\in (0,1/2)$ and
every $\gb \le 1$, so that $-h_c(\gb)/\gb^2 \in (b, 1/2-a)$.
We recall also that the (strict) inequality $h_c(\gb)<  h_c(0)$ has been
established in great generality in \cite{cf:AS}.

\medskip

In the case $\ga\in(1/2,1)$ we have the following:
\medskip

\begin{theorem}
  \label{th:a121} Let $\alpha\in(1/2,1)$. For every $\gep>0$ there
exists $a(\gep)>0$ such that 
\begin{equation}
\label{eq:witheps}
  h_c(\gb)-h_c^{ann}(\gb)\ge a(\gep)\,\gb^{(2\alpha/(2\alpha-1))+\gep},
\end{equation}
for $\beta\le1$. Moreover, $h_c(\gb)>h_c^{ann}(\gb)$ for every 
$\gb>0$.
\end{theorem}
\medskip

To appreciate this result, recall that in \cite{cf:Ken,cf:T_cmp} it
was proven that
\begin{equation}
\label{eq:fromKA-FT}
    h_c(\gb)-h_c^{ann}(\gb)\le \tilde L(1/\gb)\gb^{2\alpha/(2\alpha-1)},
\end{equation}
for some (rather explicit, {\sl cf.} in particular \cite{cf:Ken}) slowly
varying function $\tilde L(\cdot)$. Notably,  $\tilde L(\cdot)$ is
trivial if $L(\cdot)$ is.  The conclusion of Theorem~\ref{th:a121} can
actually be strengthened and we are able to replace the right-hand
side of \eqref{eq:witheps} with $\bar
L(1/\gb)\gb^{2\alpha/(2\alpha-1)}$ with $\bar L(\cdot)$ another slowly
varying function, but on one hand $\bar L(\cdot)$ does not match the
bound in \eqref{eq:fromKA-FT} and on the other hand it is rather clear
that it reflects more a limit of our technique than the actual
behavior of the model; therefore, we  decided to present the simpler
argument leading to the slightly weaker result \eqref{eq:witheps}.

\medskip
The case $\ga=1/2$ is the most delicate, and whether quenched and
annealed critical points coincide or not crucially depends on the
slowly varying function $L(\cdot)$. In \cite{cf:Ken,cf:T_cmp} it was
proven that, whenever
\begin{equation}
\label{eq:sommaconverge}
\sum_{n\ge1}\frac{1}{n\,L(n)^2}<\infty,  
\end{equation}
there exists $\beta_0>0$ such that $h_c(\beta)=h_c^{ann}(\gb)$ for $\gb\le
\gb_0$, and that when the same sum diverges then 
$h_c(\beta)-h_c^{ann}(\gb)$ is bounded {\sl above} by some function of $\gb$
which vanishes faster than any power for $\gb\searrow0$. For instance, if
$L(\cdot)$ is asymptotically constant then
\begin{equation}
  h_c(\beta)-h_c^{ann}(\gb)\le c_1\,e^{-c_2/\gb^2},
\end{equation}
for $\gb\le1$. While we are not able to prove that quenched and
annealed critical points differ as soon as condition
\eqref{eq:sommaconverge} fails  (in particular not when $L(\cdot)$
is asymptotically constant), our method can be pushed further to prove this if
$L(\cdot)$ vanishes sufficiently fast at infinity:

\medskip

\begin{theorem}
  \label{th:a12} Assume that for every $n\in\N$
  \begin{equation}
\label{eq:Lpiccola}
    K(n)\le c\frac{n^{-3/2}}{(\log n)^\eta},
  \end{equation}
for some $c>0$ and $\eta>1/2$. Then 
for every $0<\gep<\eta-1/2$ there exists $a(\gep)>0$ such that
\begin{equation}
  \label{eq:a12}
  h_c(\beta)-h_c^{ann}(\beta)\ge a(\gep) 
\exp\left(-\frac1{\gb^{\frac{1}{\eta-1/2-\gep}}}\right).
\end{equation}
Moreover, $h_c(\gb)>h_c^{ann}(\beta)$ for every $\gb>0$.
\end{theorem}

\medskip

\subsection{Fractional moment method}
In order to introduce our basic idea and, effectively, start the proof, we need some additional notation.
We fix some $k\in\N$ and we set for $n\in\N$
\begin{equation}
  \label{eq:z}
  z_n:=e^{h+\beta\omega_n}.
\end{equation} 
Then, the following identity holds for $N\ge k$:
\begin{equation}
  \label{eq:rec2}
  Z_{N,\omega}\, =\, \sum_{n=k}^N Z_{N-n,\omega}\sum_{j=0}^{k-1}K(n-j)\,
  {z_{N-j}} Z_{N-j,N,\omega}.
\end{equation}
This is simply obtained by decomposing the partition function \eqref{eq:Z}
according to the value $N-n$ of the last point of $\tau$ which does
not exceed $N-k$ (whence the condition $0\le N-n\le N-k$ in the sum), and to the value $N-j$ of the first point of $\tau$
to the right of $N-k$ (so that $N-k<N-j\le N$).  It is important to notice that
$Z_{N-j,N,\omega}$
has the same law as $Z_{j,\go}$ and that the three random variables 
$Z_{N-n,\go}$, $z_{N-j}$ and 
$ Z_{N-j,N,\omega}$
 are independent, provided that $n\ge k$ and $j<k$.

\begin{figure}[h]
\begin{center}
\leavevmode
\epsfysize =2.3 cm
\epsfxsize =14.5 cm
\psfragscanon
\psfrag{0}[c]{$0$}
\psfrag{N}[c]{$N$}
\psfrag{N-k}[c]{\small $N-k$}
\psfrag{N-j}[c]{\small $N-j$}
\psfrag{N-n}[c]{\small $N-n$}
\psfrag{Z1}[c]{$Z_{N-n, \go}$}
\psfrag{Z2}[c]{$K(n-j) z_{N-j}$}
\psfrag{Z3}[c]{$Z_{ N-j,N, \go}$}
\epsfbox{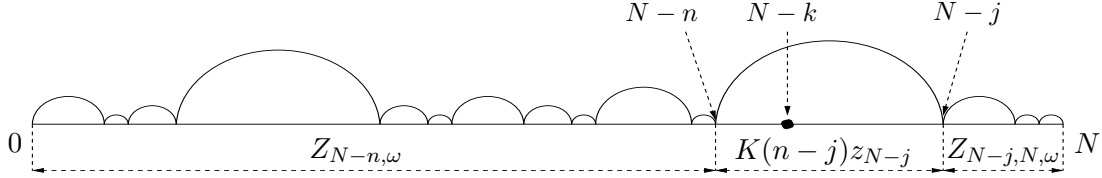}
\end{center}
\caption{\label{fig:chop} 
The decomposition of the partition function
is simply obtained by fixing a value of $k$ and 
summing over the values of the last contact (or renewal epoch) before $N-k$ and the first
after $N-k$. In the drawing the two contacts are respectively $N-n$ and $N-j$ and arcs of course  identify steps between successive contacts.}
\end{figure}

\medskip

Let $0<\gamma<1$ and $A_N:=\bbE[(Z_{N,\omega})^\gamma]$, with $A_0:=1$.
Then, from \eqref{eq:rec2} and using the elementary inequality 
\begin{equation}
  \label{eq:ineqgamma}
(a_1+\ldots+a_n)^\gamma\le
a_1^\gamma+\ldots+a_n^\gamma , 
\end{equation}
which holds for $a_i\ge0$, one deduces
\begin{equation}
  \label{eq:rec3}
  A_N\le 
\bbE[ z_1^\gamma] \sum_{n=k}^N A_{N-n}
\sum_{j=0}^{k-1}K(n-j)^\gamma A_j.
\end{equation}

The basic principle is the following:

\medskip

\begin{proposition}
\label{th:deloc}
  Fix $\beta$ and $h$. If there exists $k\in\N$ and $\gamma<1$ such that 
\begin{equation}
  \label{eq:if}
\rho\,:=\,  \bbE[ z_1^\gamma]  \sum_{n=k}^\infty 
\sum_{j=0}^{k-1}K(n-j)^\gamma A_j\le1,
\end{equation}
then $\tf(\beta,h)=0$. Moreover if
$\rho<1$ there exists $C=C(\rho, \gamma, k, K(\cdot))>0$ such that
\begin{equation}
  \label{eq:if2}
  A_N \, \le \, C \left( K(N)\right)^\gamma,
  \end{equation}
  for every $N$.
\end{proposition} 
\medskip

%%G
Of course, in view of the results we want to prove, 
the main result of Proposition~\ref{th:deloc} is the first one.
The second one, namely \eqref{eq:if2}, is however of independent
interest and may be used to obtain path estimates on the process
(using for example the techniques in \cite{cf:GTdeloc} and
\cite[Ch.~8]{cf:GB}).

\medskip

{\sl Proof of Proposition \ref{th:deloc}}.  Let
$\bar A:=\max\{A_0,A_1,\ldots,A_{k-1}\}$. From \eqref{eq:rec3} it
follows that for every $N\ge k$
\begin{equation}
  A_N\, \le \, \rho \max\{A_0,\ldots,A_{N-k}\},
\end{equation}
from which one sees by induction that, since $\rho\le 1$, for every
$n$ one has $A_n\le \bar A$.  The statement $\tf(\beta,h)=0$ follows
then from Jensen's inequality:
\begin{equation}
  \tf(\beta,h)=\lim_{N\to\infty}\frac1{N\gamma}\bbE \log (Z_{N,\go})^\gamma\le 
\lim_{N\to\infty}\frac1{N\gamma}\log A_N=0.
\end{equation}

%%G

In order to prove \eqref{eq:if2} we introduce
\begin{equation}
Q_k(n)\, :=\, \begin{cases}
\bbE[ z_1^\gamma]
\sum_{j=0}^{k-1}K(n-j)^\gamma A_j,
&\text{ if } n \ge k,
\\
0 &\text{ if } n=1, \ldots k-1.
\end{cases}
\end{equation}
Since $\rho= \sum_n Q_k(n)$, the assumption $\rho<1$
tells us that $Q_k(\cdot)$ is a sub-probability distribution
and it becomes a probability distribution if we set, as we do, $Q_k(\infty):=
1-\rho$. Therefore the renewal process $\tilde \tau$ with 
inter-arrival law $Q_k(\cdot)$ is {\sl terminating}, that is 
$\tilde \tau$ contains, almost surely, only a finite number of points.
A particularity of terminating renewals with regularly varying
inter-arrival  distribution is  the asymptotic equivalence, up
to a multiplicative factor, of inter-arrival distribution and mass renewal function
(\cite[Th.~A.4]{cf:GB}), namely
\begin{equation}
\label{eq:3stars}
{u_N}\, \stackrel{N\to \infty} \sim\, \frac 1{(1-\rho)^2} {Q_k(N)},
\end{equation}
where $u_N:= \bP (N \in \tilde \tau)$ and it satisfies the 
renewal equation $u_N= \sum_{n=1}^N u_{N-n} Q_k(n)$
for $N\ge 1$ (and $u_0=1$). Since $Q_k(n)=0$ for 
$n=1, \ldots, k-1$, for the same values of $n$ we have 
$u_n=0$ too. Therefore the renewal equation may be rewritten,
for $N \ge k$,
as
\begin{equation}
\label{eq:1star}
u_N\, =\, \sum_{n=1}^{N-k} u_{N-n} Q_k(n) \, +\,  Q_k(N).
\end{equation}

Let us observe now that if we set $\tilde A_N := A_N \ind _{N \ge k}$
then \eqref{eq:rec3} implies that for $N \ge k$
\begin{equation}
\tilde A_N \, \le \,  \sum_{n=1}^{N-k}\tilde A_{N-n} Q_k(n) \, +\,
P_k(N), \ \ \ \text{ with } \ \ 
P_k(N)\, :=\, \sum_{n=0}^{k-1} A_n Q_k(N-n),
\end{equation}
and observe  that $P_k(N) \le c \, Q_k(N)$, with
$c$ that depends on $\rho$, $\gamma$, $k$ and $K(\cdot)$
(and on $h$ and $\gb$, but these variables  are kept fixed).
Therefore
\begin{equation}
\label{eq:2stars}
\frac{\tilde A_N}c \, \le \,  \sum_{n=1}^{N-k}\frac{\tilde A_{N-n}}c Q_k(n) \, +\,
\, Q_k(N),
\end{equation}
for $N \ge k$. 
By comparing \eqref{eq:1star} and \eqref{eq:2stars}, and
by using \eqref{eq:3stars} and 
$Q_k(N) \stackrel{N \to \infty} \sim K(N)^\gamma
\bbE[ z_1^\gamma]
\sum_{j=0}^{k-1} A_j$, 
 one directly obtains  \eqref{eq:if2}.
\qed

\subsection{Disorder relevance: sketch of the proof}
Let us consider for instance the case $\ga>1$, which is technically
less involved than the others, but still fully representative of
our strategy.  Take $(\gb,h)$ such that $\gb$ is small and
$h=h_c^{ann}(\gb)+\Delta$, with $\Delta=a\gb^2$. We are therefore
considering the system inside the annealed localized phase, but close
to the annealed critical point (at a distance $\Delta$ from it), and
we want to show that $\tf(\gb,h)=0$. In view of Proposition
\ref{th:deloc}, it is sufficient to show that $\rho$ in
\eqref{eq:if} is sufficiently small, and we have the freedom to choose
a suitable $k$.  Specifically, we choose $k$ to be of the order of the
correlation length of the annealed system:
$k=1/\tf^{ann}(\gb,h)=1/\tf(0,\Delta)\approx \textrm{const.}/(a\gb^2)$, where the
last estimate holds since the phase transition of the annealed system
is first order for $\ga>1$. Note that $k$ diverges for $\gb$ small.

For the purpose of this informal discussion, assume that
$K(n)=c\,n^{-(1+\alpha)}$, {\sl i.e.}, the slowly varying function 
$L(\cdot)$ is constant. The sum over $n$ in the right-hand side 
of \eqref{eq:if}
%%FF
is then immediately performed
and (up to a  multiplicative constant) one is left with estimating
\begin{equation}
\label{eq:schecci}
  \sum_{j=0}^{k-1}\frac {A_j}{(k-j)^{(1+\ga)\gga-1}}.
\end{equation}

One can choose $\gamma<1$ such $(1+\ga)\gga-1>1$ and it is actually
not difficult to show that $\sup_{j<k}A_j$ is bounded by a constant
uniformly in $k$. On one hand in fact $A_j\le [\bbE
Z_{j,\go}]^\gamma=[Z_j(\Delta)]^\gamma$, where the first step follows
from Jensen's inequality and the second one from the definition of the
model (recall \eqref{eq:Zpure}). On the other hand for $j< k$, {\sl
  i.e.}, for $j$ smaller than the correlation length of the annealed
model, one has that the annealed partition function $Z_{j}(\Delta)$ is
bounded above by a constant, {\sl independently of how small $\Delta$
  is}, {\sl i.e.}, of how large the correlation length is.  This
just establishes that the quantity in \eqref{eq:schecci} is bounded,
so we need to go beyond and show that $A_j$ is small: this of course
is not true unless $j$ is large, but if we restrict the sum in
\eqref{eq:schecci} to $j\ll k$ what we obtain is small, since the
denominator is approximately $k^{(1+\ga)\gamma -1}$, that is $k$ to a power
larger than $1$.

In order to control the terms for which $k-j$ is of order $1$ a new
ingredient is clearly needed, and we really have to estimate the
fractional moment of the partition function without resorting to
Jensen's inequality. To this purpose, we apply an idea which was
introduced in \cite{cf:GLT}. Specifically, we change the law $\bbP$ of
the disorder in such a way that under the new law, $\tilde \bbP$, the
system is delocalized and $\tilde\bbE(Z_{j,\go})^\gamma$ is small.
The change of measure corresponds to tilting negatively the law of
$\go_i, i\le j$, {\sl cf.}  \eqref{eq:tildeP}, so that the system is more
delocalized than under $\bbP$.  The non-trivial fact is that with our
choice $\Delta=a\gb^2$ and $j\le 1/\tf(0,\Delta)$, one can guarantee
on one hand that $Z_{j,\go}$ is typically small under $\tilde\bbP$,
and on the other that $\bbP$ and $\tilde \bbP$ are close (their mutual
density is bounded, in a suitable sense), so that the same statement
about $Z_{j,\go}$ holds also under the original measure $\bbP$. At
this point, we have that all terms in \eqref{eq:schecci} are small:
actually, as we will see, the whole sum is as small as we wish if we
choose $a$ small.  The fact that $\tf(\gb,h)=0$ then follows from
Proposition \ref{th:deloc}.

As we have mentioned above, the case $\ga \in [1/2, 1)$ is not much 
harder, at least on a conceptual level, but this time it is not
sufficient to establish bounds on  $A_j$ that do not depend
on $j$: the exponent in the denominator of the summand in  \eqref{eq:schecci}
is in any case smaller than $1$ and one has to exploit
the decay in $j$ of $A_j$: with respect to the 
 $\ga>1$ case, here one can exploit the decay of  $\bP (j \in \tau)$
 as $j$ grows, while such a quantity converges to a positive constant if
 $\ga >1$. 
Once again the case of $j \ll k$ can be dealt with by direct 
annealed estimates, while when one gets close to $k$ 
a finer argument, direct generalization of the one used for
the $\ga>1$ case, is needed.

%%G 

%We mention that a proof of non-coincidence of quenched and annealed critical points in the case $\alpha=1$ and $\bE(\tau_1)=\sum_n L(n)/n=\infty$ has been recently announced by K. Alexander and N. Zygouras \cite{cf:Ken_eurandom}. Making our method work in the $\ga=1$ case appears to be technically harder than the cases we are considering in this work.

\section{The case $\alpha>1$}

In order to avoid repetitions let us establish that, 
in this and next sections, $R_i, i=1,2,\ldots$  denote (large)
constants,  $L_i(\cdot)$ are slowly varying functions and $C_i$ 
positive constants (not necessarily large).

\medskip 

{\sl Proof of Theorem \ref{th:a>1}}. Fix $\beta_0>0$ 
 and let $\beta\le \beta_0$, 
$h=h_c^{ann}(\beta)+ a \beta^2$ and $\gamma<1$ sufficiently close to
$1$ so that
\begin{equation}
\label{eq:alphans}
  (1+\alpha)\gamma>2.
\end{equation}
It is sufficient to show that the sum in \eqref{eq:if} can be made
arbitrarily small (for some suitable choice of $k$) by choosing $a$
small, since $\bbE [z_1^\gamma]$ can be bounded above by a constant
independent of $a$ (for $a$ small).

We choose $k=k(\beta)=1/(a\beta^2)$, so that $\beta=1/\sqrt{a
  k(\beta)}$.  In order to avoid a plethora of $\lfloor\cdot\rfloor$,
we will assume that $k(\beta)$ is integer.  Note that $k(\beta)$ is
large if $\beta$ or $a$ are small.

First of all note that, thanks to Eqs. \eqref{eq:sommasv} and 
\eqref{eq:maxsv}, the sum in the r.h.s. of \eqref{eq:if} is
bounded above by
\begin{equation}
\label{eq:S}
\sum_{j=0}^{k(\gb)-1}
\frac{L_1(k(\gb)-j)\,A_j}{(k(\gb)-j)^{(1+\ga)\gamma-1}}.
\end{equation}

We split this sum as
\begin{equation}
\label{eq:split1}
  S_1+S_2:=  \sum_{j=0}^{k(\gb)-1-R_1}
\frac{L_1(k(\gb)-j)\,A_j}{(k(\gb)-j)^{(1+\ga)\gamma-1}}
+ \sum_{j=k(\gb)-R_1}^{k(\gb)-1}
\frac{L_1(k(\gb)-j)\,A_j}{(k(\gb)-j)^{(1+\ga)\gamma-1}}.
\end{equation}
To estimate $S_1$, note that by Jensen's inequality $A_j\le (\bbE
Z_{j,\go})^\gamma\le C_1$ with $C_1$ a constant independent of $j$ as
long as $j< k(\gb)$.
Indeed, from \eqref{eq:Z} and the definition of the annealed critical
point one sees that (recall \eqref{eq:Zpure})
\begin{equation}
\label{eq:C0}
  \bbE{Z_{j,\go}}=Z_j(a\beta^2)=\bE\left(e^{a\beta^2|\tau\cap\{1,\ldots,j\}|}
\ind_{\{j\in\tau\}}\right),
\end{equation}
and the last term is clearly smaller than $e$.
Therefore, using again \eqref{eq:sommasv}
\begin{equation}
  S_1\le 
\frac{L_2(R_1)}{R_1^{(1+\ga)\gga-2}},
\end{equation}
which can be made small with $R_1$ large in view of the 
choice \eqref{eq:alphans}.  As for $S_2$, one has
\begin{equation}
\label{eq:B2}
  S_2\, \le\, 
C_2\,\max_{k(\gb)-R_1\le j<k(\gb)}A_j.
\end{equation}
We apply now Lemma~\ref{th:bernard} (note also the definition in
 \eqref{eq:tildeP}) with $N=j$ and
$\gl= 1/\sqrt{j}$ 
so that we have 
\begin{equation}
A_j \, \le \, \left[\bbE_{j, 1/\sqrt{j}} \left( Z_{j, \go}\right)\right]^\gamma
\exp\left(c \gamma/(1-\gamma)\right), 
\end{equation}
for $1/\sqrt{j} \le \min (1, (1-\gamma)/\gamma)$, that is for $a$
sufficiently small, since we are in any case assuming $j\ge
k(\gb)-R_1$.

We are therefore left with showing that $ \bbE_{j, 1/\sqrt{j}} \left[
  Z_{j, \go}\right]$ is small for the range of $j$'s we are
  considering.  For such an estimate it is convenient to recall
  \eqref{eq:hann} and to observe that for any given values of $\beta$,
  $ h$ and $\gl$ and for any $j$
\begin{equation}
\label{eq:MMM}
\bbE_{j,\gl}[Z_{j,\go}]=\bE
\left[\left(\exp\left({h-h_c^{ann}(\gb)}\right)
\frac{\M(\gb-\gl)}{\M(\gb)\M(-\gl)}
  \right)^{|\tau\cap\{1,\ldots,j\}|}\;\ind_{\{j\in\tau\}}\right].
\end{equation}
In order to exploit such a formula let us observe that
\begin{equation}
\label{eq:MM}
   \frac{\M(\gb-\gl)}{\M(\gb)\M(-\gl)}=
\exp\left[-\int_0^\gb \dd x\int_{-\gl}^0\dd y \left.\frac{\dd^2}{\dd t^2}
\log \M(t)\right|_{t=x+y}\right]
\, \le\,  e^{-C_3\beta\gl}, 
\end{equation}
which holds for $0 < \gl \le  \gb \le \gb_0$  and $C_3:=
\min_{t \in [-\gb_0, \gb_0]} \dd^2(\log \M (t))/\dd t^2>0$.
  If $a$ is sufficiently small, for
$j \le k(\gb)=1/(a\gb^2)$ we have
\begin{equation}
a\beta^2-\frac{C_3\gb}{\sqrt{j}}\le\frac1{k(\gb)}\left[1-\frac{C_3}{\sqrt
    a}\right] \le -\frac{C_3}{2k(\gb)\sqrt a}.
\end{equation}
As a consequence,
\begin{equation}
\label{eq:law}
\max_{k(\gb)-R_1\le j<k(\gb)}
\bbE_{j, 1/\sqrt{j}}(Z_{j,\go})\,\le\, 
e^{C_3\sqrt a \gb^2 R_1/2}\,
\bE\left[\exp\left(-\frac{C_3}{2\sqrt{a}k(\gb)}
\left|\tau\cap\{1,\ldots,k(\gb)\}\right|
  \right)
\right].
\end{equation}
The right-hand side in  \eqref{eq:law} can be made small
by choosing $a$ small (and this is uniform on $\beta\le \beta_0$) because of
\begin{equation}
  \label{eq:isi}
  \lim_{c\to+\infty}\limsup_{N\to\infty} 
\bE\left(e^{-(c/N)|\tau\cap\{1,\ldots,N\}|}\right)=0,
\end{equation}
that we are going to prove just below.
Putting everything together, we have shown that both $S_1$ and $S_2$
can be made small via a suitable choice of $R_1$ and $a$, and the
theorem is proven.

To prove \eqref{eq:isi}, since the function under
expectation is bounded by $1$ it is sufficient to observe that
\begin{equation}
\frac1N
\sum_{n=1}^N\ind_{\{n\in\tau\}}\stackrel{N\to\infty}\longrightarrow
\frac1{\sum_{n\in\N}n K(n)}=\frac1{\bE(\tau_1)}\, >\,0,
\end{equation}
almost surely (with respect to $\bP$) by the classical Renewal
Theorem (or by the strong law of large numbers).

The claim $h_c(\gb)>h_c^{ann}(\gb)$ for every $\gb$ follows from the 
arbitrariness of $\beta_0$.
\qed

\section{The case $1/2<\alpha<1$}

{\sl Proof of Theorem \ref{th:a121}}.  To make things clear, we fix
now $\gep>0$ small 
and $0<\gga<1$ such that
\begin{equation}
\label{eq:gga1}
\gga\left\{(1+\ga)+(1-\gep^2)\left[1-\alpha+(\gep/2)(\ga-1/2)\right]\right\}
\, > \, 2,
  \end{equation}
  and
  \begin{equation}
   \label{eq:gga2}
 \gga\left[(1+\ga)+(1-\gep^2)(1-\ga)\right]\, >\, 2-\gep^2.
 \end{equation}
Moreover we  take $\gb\le \beta_0$ and
\begin{equation}
\label{eq:h}
h=h_c^{ann}(\gb)+\Delta:=
h_c^{ann}(\gb)+a\gb^{\frac{2\ga}{2\ga-1}(1+\gep)}.
\end{equation}
We notice that it is crucial that $(\ga-1/2)>0$ for \eqref{eq:gga1} to
be satisfied.  We will take $\gep$ sufficiently small (so that
\eqref{eq:gga1} and \eqref{eq:gga2} can occur) {\sl and then, once
  $\gep$ and $\gga$ are fixed,} $a$ also small.  We set moreover
\begin{equation}
  \label{eq:kb}
  k(\gb):= \frac1{\tf(0,\Delta)}
\end{equation}
and we notice that $k(\gb)$ can be made large by choosing $a$
small, uniformly for $\gb\le\beta_0$. As in the previous section, we
assume for ease of notation that $k(\gb)\in\N$ (and we  write just
$k$ for $k(\gb)$). 

Our aim is to show that $\tf(\gb,h)=0$ if $a$ is chosen
sufficiently small in \eqref{eq:h}.  We recall that, thanks to
Proposition \ref{th:deloc}, the result is proven if we show that
\eqref{eq:S} is $o(1)$ for $k$ large.  
In order to estimate this sum, we need a couple of technical
estimates which are proven at the end of this section (Lemma \ref{th:lemmaA})
and in Appendix \ref{sec:Arenew} (Lemma \ref{th:lemmaU2}).
\begin{lemma}
\label{th:lemmaU2} Let $\ga\in(0,1)$.
There exists a constant $C_4$ such that for every
$0<h<1$ and every 
$j\le 1/\tf(0,h)$
  \begin{equation}
     Z_{j}(h)\, \le\, \frac{ C_4}{j^{1-\ga}L(j)}.
  \end{equation}
\end{lemma}
In view of $Z_j(h_c(0))=Z_j(0)=\bP(j\in\tau)$ and \eqref{eq:DoneyB},
this means that as long as $j\le 1/\tf(0,h)$ the partition function of
the homogeneous model behaves essentially like in the (homogeneous)
critical case.  \medskip

\begin{lemma}
\label{th:lemmaA}
  There exists $\gep_0>0$ such that, if $\gep\le \gep_0$ ($\gep$ 
being the same one which appears in \eqref{eq:h}),
  \begin{equation}
    \bbE_{j,1/\sqrt{j}}[Z_{j,\go}]\le
    \frac{C_5}{j^{1-\ga+(\gep/2)(\ga-1/2)}}
  \end{equation}
  for some constant $C_5$ (depending on $\gep$ but not on $\gb$ or $a$),
  uniformly in $0\le \gb\le\gb_0$ and in $k^{(1-\gep^2)}\le j<k$.
\end{lemma}

\medskip

In order to bound above \eqref{eq:S}, we split it as
\begin{equation}
  S_3+S_4:=\sum_{j=0}^{\lfloor k^{(1-\gep^2)}\rfloor}
  \frac{L_1(k-j)\,A_j}{(k-j)^{(1+\ga)\gamma-1}}+
\sum_{j=\lfloor k^{(1-\gep^2)}\rfloor+1}^{k-1}
  \frac{L_1(k-j)\,A_j}{(k-j)^{(1+\ga)\gamma-1}}.
\end{equation}
For $S_3$ we use simply $A_j\le (\bbE
Z_{j,\go})^\gamma=[Z_j(\Delta)]^\gga$ and Lemma \ref{th:lemmaU2}, 
together with \eqref{eq:sommasv} and \eqref{eq:maxsv}:
\begin{equation}
\label{eq:S3}
  S_3\le \frac{L_3(k)}{k^{[(1+\ga)\gamma-1]}}\frac1{
k^{(1-\gep^2)((1-\ga)\gamma-1)}},
\end{equation}
where $L_3(\cdot)$ can depend on $\gep$ but not on $a$.  The second
condition \eqref{eq:gga2} imposed on $\gga$ guarantees that $S_3$ is
arbitrarily small for $k$ large, {\sl i.e.}, for $a$ small.

As for $S_4$, we use Lemma \ref{th:bernard} with $N=j$ and
$\gl=1/\sqrt{j}$ to estimate $A_j$ (recall the definition in
\eqref{eq:tildeP}). We get
\begin{align}
  A_j\le\left[\bbE_{j,1/\sqrt{j}}(Z_{j,\go})\right]^{\gga}\exp(c\gga/(1-\gga)),
  \label{eq:esti}
\end{align}
provided that 
%%F
$1/\sqrt j\le \min(1,(1-\gga)/\gga)$, which is true for all
$j\ge k^{1-\gep^2}$ if $a$ is small.  Then, provided we have chosen
$\gep\le \gep_0$, Lemma \ref{th:lemmaA} gives for every
$k^{(1-\gep^2)}< j<k$,
\begin{equation}
   A_j\le  \frac{C_{6}}{j^{[1-\ga+(\gep/2)(\ga-1/2)]\gamma}}.
\end{equation}
Note that $C_6$ is large for $\gep$ small (since from
\eqref{eq:gga1}-\eqref{eq:gga2} it is clear that $\gga$ must be close to 
$1$ for $\gep$ small) but it is independent of $a$. As a consequence,
using \eqref{eq:sommasv2},
\begin{equation}
\begin{split}
  S_4&\, \le\, \max_{k^{(1-\gep^2)}\le j<k}A_j\times
\sum_{r=1}^{k}\frac{L_1(r)}{r^{(1+\ga)\gga-1}}\le
\max_{k^{(1-\gep^2)}\le j<k}A_j\times
\frac{L_4(k)}{k^{(1+\ga)\gga-2}}\\
&\, \le\, C_{6}\,L_4(k)\,k^{2-(1+\ga)\gga-
(1-\gep^2)[1-\ga+(\gep/2)(\ga-1/2)]\gga}.
\end{split}
\end{equation}
Then, the first condition \eqref{eq:gga1} imposed on $\gga$  guarantees
that $S_4$ tends to zero when $k$ tends to infinity. 
\qed

\bigskip

{\sl Proof of Lemma \ref{th:lemmaA}.}  Using \eqref{eq:MMM} together
with the observation \eqref{eq:MM}, the definition of $\Delta$ and
of $k=k(\gb)$ in terms of $\tf(0,\Delta)$ (plus the 
behavior of $\tf(0,\Delta)$ for $\Delta$ small described
in Theorem \ref{th:pure} (2)) one sees that for $j\le k(\gb)$
\begin{equation}
\label{eq:sole11}
  \bbE_{j,1/\sqrt{j}}[Z_{j,\go}]\le \bE\left(
e^{-C_{7}\frac{\gb}{\sqrt j}|\tau\cap
\{1,\ldots,j\}|}\,\ind_{\{j\in\tau\}}\right),
\end{equation}
uniformly for $0\le \gb\le\beta_0$. If moreover $j\ge k^{(1-\gep^2)}$
one has
\begin{equation}
\frac{\gb}{\sqrt j}\ge \frac{C_{8}}{j^{1/2+(\ga-1/2)(1+2
\gep^2)/(1+\gep)}}\ge  \frac{C_{8}}{j^{\ga-(\gep/2)(\ga-1/2)}},
\end{equation}
with $C_{8}$ independent of $a$ for $a$ small.
The condition that $\gep$ is small has been used, say, to neglect 
$\gep^2$ with respect to $\gep$.
Going back to \eqref{eq:sole11} and using Proposition \ref{th:homog}
one has then
\begin{equation}
  \bbE_{j,1/\sqrt{j}}[Z_{j,\go}]\le 
  \frac{C_{9}}{j^{1-\ga+(\gep/2)(\ga-1/2)}}.
\end{equation}
with $C_{9}$ depending on $\gep$ but not on $a$.
\qed

\section{The case $\ga=1/2$}

{\sl Proof of Theorem \ref{th:a12}.}  The proof is not conceptually
different from that of Theorem \ref{th:a121}, but here we have to
carefully keep track of the slowly varying functions, and we have to
choose $\gamma(<1)$ as a function of $k$.  Under our assumption
\eqref{eq:Lpiccola} on $L(\cdot)$, it is easy to deduce from Theorem
\ref{th:pure} (2) that (say, for $0<\Delta<1$)
 \begin{equation}
 \label{eq:F1/2}
   \tf(0,\Delta)=\Delta^2\hat L(1/\Delta)\ge C(c,\eta)\Delta^2\,|\log 
   \Delta|^{2\eta}.
 \end{equation}
We take $\gb\le\gb_0$ and 
\begin{equation}
\label{eq:D12}
h\, =\, h_c^{ann}(\gb)+\Delta\, :=\,  h_c^{ann}(\gb)+a
\exp\left(-\gb^{-1/(\eta-1/2-\gep)}\right),
\end{equation}
 and, as in last section, $k=1/\tf(0,\Delta)=\Delta^{-2}/\hat L(1/\Delta)$.
We note also that (for $a<1$) 
\begin{equation}
\label{eq:gbge}
  \gb\ge |\log \Delta|^{-\eta+1/2+\gep}.
\end{equation}
We set $\gga=\gga(k)=1-1/(\log k)$. As $\gga$ is $k$--dependent one
cannot use \eqref{eq:sommasv} and \eqref{eq:maxsv} without care to pass 
from \eqref{eq:if} to \eqref{eq:S}, since one could in principle
have  $\gamma$-dependent (and therefore $k$-dependent) constants
in front.
Therefore, our first aim will be to (partly) get rid of $\gga$ in
\eqref{eq:if}. We notice that for any $j\le k-1$, for $k$ such that
$\gga(k)\ge 5/6$,
\begin{align}
\label{eq:tispiezzoindue}
  \sum_{n=k}^{\infty}K(n-j)^{\gamma}&\le
  \sum_{n=k-j}^{k^6}K(n)\exp\left[(3/2\log n-\log L(n))/\log
    k\right]+\sum_{n=k^6+1}^{\infty}[K(n)]^{5/6}.
\end{align}
Now, properties of slowly varying functions guarantee that the quantity
in the exponential in the first sum is bounded (uniformly in $j$ and
$k$). As for the second sum, \eqref{eq:sommasv} guarantees it is
smaller than $k^{-6/5}$ for $k$ large. Since by Lemma \ref{th:lemmaU2}
the $A_j$ are bounded by a constant in the regime we are considering,
when we reinsert this term in \eqref{eq:if} and we sum over $j<k$ we
obtain a contribution which vanishes at least like $k^{-1/5}$ for
$k\to\infty$. We will therefore forget from now on the second sum in
\eqref{eq:tispiezzoindue}.

Therefore one has
\begin{align}
\rho  \le C_{10}\sum_{n=k}^{\infty}\sum_{j=0}^{k-1}K(n-j)A_j
      \le C_{11}\sum_{j=0}^{k-1}\frac{L(k-j)A_j}{(k-j)^{1/2}},
\end{align}     
where we have safely used \eqref{eq:sommasv} to get the second expression
and now $\gamma$ appears only (implicitly) in the fractional moment
$A_j$ but not in the constants $C_i$.

Once again, it is convenient to split this sum into
\begin{equation} 
  S_5+S_6:=\sum_{j=0}^{k/R_2}\frac{A_j\, L(k-j)}{ (k-j)^{1/2}}
+\sum_{j=(k/R_2)+1}^{k-1}\frac{A_j\,L(k-j)}{(k-j)^{1/2}},
\end{equation}
with $R_2$ a large constant.  To bound $S_5$ we simply use Jensen
inequality to estimate $A_j$. Lemma \ref{th:lemmaU2} gives that for
all $j\le k$,
\begin{align}
A_j\le \frac{C_{12}}{j^{\gga/2}L(j)^\gga}\le \frac{C_{13}}{\sqrt j L(j)},
\end{align}
where the second inequality comes from our choice $\gamma=1-1/(\log
k)$. Knowing this, we can use \eqref{eq:sommasv} to compute $S_5$ and
get
%%F
 \begin{equation}
\label{eq:regvar}
S_5\le \frac{C_{14}}{\sqrt{R_2}}\frac{L(k(1-1/R_2))}
{L(k/R_2)}.
\end{equation}
We see that $S_5$ can be made small
choosing $R_2$ large. It is important for the following to note that
it is sufficient to choose $R_2$ large but independent of $k$; in
particular, for $k$ large at $R_2$ fixed the last factor in
\eqref{eq:regvar} approaches $1$ by the property of slow variation of
$L(\cdot)$. As for $S_6$,
\begin{equation}
  S_6\le C_{15}\max_{k/R_2< j<k}A_j\times \sqrt k\,L(k).
\end{equation}
In order to estimate this maximum, we need to refine Lemma \ref{th:lemmaA}:
\begin{lemma}
  \label{th:lemmaa12} 
There exists a constant
%%F
$C_{16}:=C_{16}(R_2)$ such that for  $\gamma=1-1/(\log k)$ and $k/R_2< j<k$
  \begin{equation}
    A_j\le C_{16}\left(L(j)\sqrt j\,(\log j)^{2\gep}\right)^{-1}.
  \end{equation}
\end{lemma}
Given this, we obtain immediately
%%F
\begin{equation}
  S_6 \, \le\,  C_{17}(R_2)\left[\log
\left(\frac k{R_2}\right)\right]^{
-2\gep}.
\end{equation}
It is then clear that $S_6$ can be made
arbitrarily small with $k$ large, {\sl i.e.}, with $a$ small.
\qed

\medskip

{\sl Proof of Lemma \ref{th:lemmaa12}.}
Once again, we use Lemma \ref{th:bernard} with $N=j$ but this time $\gl=(j\log
j)^{-1/2}$. Recalling that $\gga=1-1/(\log k)$  we obtain
\begin{align}
  A_j\le \left[\bbE_{j,(j\log
      j)^{-1/2}}(Z_{j,\go})\right]^\gga\exp\left(c\frac{\log k}{\log
      j}\right), \label{eq:undemi}
\end{align}
for all $j$ such that $(j\log j)^{1/2} \ge \log k$. The latter
condition is satisfied for all $k/R_2<j<k$ if $k$ is large enough.
Note that, since $j>k/R_2$, the exponential factor in \eqref{eq:undemi} is
bounded by a constant 
%%F
$C_{18}:=C_{18}(R_2)$.

Furthermore, for $j\le k$, Eqs. \eqref{eq:MMM}, \eqref{eq:MM} combined give
\begin{align}\label{eq:ss}
\bbE_{j,(j\log j)^{-1/2}}[Z_{j,\go}]\le Z_j\left(-C_{19}\gb(j\log j)^{-1/2}\right),
\end{align}
for some positive constant $C_{19}$, provided $a$ is small (here we
have used \eqref{eq:F1/2} and the definition $k=1/\tf(0,\gD)$).

In view of $j\ge k/R_2$, the definition of $k$ in terms of $\gb$ and 
assumption \eqref{eq:Lpiccola},
we see that 
\begin{equation}
\gb\ge C_{20}(\log j)^{(-\eta+1/2+\gep)}
\ge \frac {C_{21}}c L(j)(\log j)^{1/2+\gep},
\end{equation} 
so that the r.h.s. of \eqref{eq:ss} is bounded above by
\begin{equation}
  Z_j\left(-C_{21}\frac{L(j)}{c\sqrt{j}}(\log j)^{\gep}\right)
\le C_{22}\frac{(\log j)^{-2\gep}}{L(j)\,\sqrt j} ,
\end{equation}
where in the last inequality we used Lemma \ref{th:homog}. The result
is obtained by re-injecting this in \eqref{eq:undemi}, and using the
value of $\gamma(k)$.

\qed

\appendix

\section{Frequently used bounds}
\label{sec:bounds}

\subsection{Bounding the partition function via tilting}
For $\gl \in \R$ and $N \in \N$ consider the probability measure
$\bbP_{N, \gl}$ defined by 
\begin{equation}
\label{eq:tildeP}
\frac{
\dd \bbP_{N, \gl}}{\dd \bbP} (\go)\, =\,
\frac1 {\M(-\gl)^N} 
\exp\left(-\gl \sum_{i=1}^N \go_i \right),
\end{equation} 
where $\M(\cdot)$ was defined in \eqref{eq:M}.
Note that under $\bbP_{N,\gl}$ the random variables $\go_i$ are
still independent but no longer identically distributed: the law of
$\go_i, i\le N$ is tilted while $\go_i, i>N$ are distributed exactly
like under $\bbP$.

\medskip

\begin{lemma}
\label{th:bernard}
There exists $c>0$ such that, for every $N \in \N $ and $\gamma \in
(0,1)$,
\begin{equation}
\label{eq:bernard}
\bbE \left[ \left(Z_{N, \go}\right)^\gamma 
\right]\, \le \,
\left[\bbE_{N , \gl} \left(Z_{N, \go}\right)
\right]^\gamma
\, \exp \left( c \left( \frac{\gga}{1-\gamma} \right)\gl ^2 N
\right),
\end{equation}
for $\vert \gl \vert \le \min(1,(1-\gamma)/\gamma)$. 
\end{lemma}
\medskip

\noindent
{\it Proof.}
 We have
\begin{equation}
\label{eq:Holder}
\begin{split}
\bbE \left[ \left(Z_{N, \go}\right)^\gamma 
\right]\, &=\, 
\bbE_{N , \gl} \left[ \left(Z_{N, \go}\right)^\gamma 
\frac{\dd \bbP}{\dd \bbP_{N , \gl}}(\go)
\right]
\\
&\le \,
 \left[\bbE_{N , \gl} \left(Z_{N, \go}\right)
\right]^\gamma
\left(
\bbE_{N,\gl}\left[ \left(\frac{\dd \bbP}{\dd \bbP_{N , \gl}}
(\go)\right)^{1/(1-\gamma)}\right]\right)^{1-\gamma}
\\
&=\, \left[\bbE_{N , \gl} \left(Z_{N, \go}\right)
\right]^\gamma
\left( \M (-\gl)^\gga \M\left( \gl \gga  /(1-\gamma)\right)^{1-\gamma}\right)^N, 
\end{split}
\end{equation}
where in the second step we have used 
H\"older inequality and the last step is a direct computation.
The proof is complete once we observe that
$0 \le \log \M (x) \le  c x^2$ for $\vert x\vert \le 1$
if $c$ is the maximum of the   second derivative of $(1/2)\log \M (\cdot)$
over $[-1, 1]$.

\qed

\subsection{Estimates on the renewal process}
\label{sec:Arenew}

With the notation \eqref{eq:Zpure} one has
\medskip

\begin{proposition}\label{th:homog} Let $\alpha\in(0,1)$ and 
 $r(\cdot)$ be a function diverging at infinity and such that
\begin{equation}
\label{eq:condr}
\lim_{N\to\infty}\frac{r(N)L(N)}{N^{\ga}}=0.
\end{equation}
For the homogeneous pinning model,
\begin{equation}
Z_{N}(-N^{-\ga}L(N)r(N))\stackrel{N\to\infty}\sim
\frac{N^{\ga-1}}{L(N)\,r(N)^2}.
\end{equation}
\end{proposition}
\medskip

To prove this result we use:
\medskip

\begin{proposition}(\cite[Theorems A \& B]{cf:Doney})\label{th:Don}
  Let $\ga\in(0,1)$.  There exists a function $\gs(\cdot)$ satisfying
\begin{equation}
\lim_{x\rightarrow+\infty} \gs(x)\, =\, 0 ,
\end{equation}
and such that 
for all $n,N\in\N$
\begin{equation}
\label{eq:fD}
\left|\frac{\bP(\tau_n=N)}{n K(N)}-1\right|\le \gs\left(\frac{N}{a(n)}\right),
\end{equation}
where $a(\cdot)$ is an asymptotic inverse of $x \mapsto x^\ga /L(x)$.

Moreover, 
\begin{equation}
  \label{eq:DoneyB}
  \bP(N\in\tau)\stackrel{N\to\infty} \sim 
  \left(
  \frac{\ga \sin(\pi \ga)}{\pi}\right)
   \frac{N^{\ga-1}}{L(N)}.
\end{equation}
\end{proposition}
\medskip

We observe that by \cite[Th.~1.5.12]{cf:RegVar}
we have that $a(\cdot)$ is regularly varying of exponent 
$1/\ga$, in particular $\lim_{n \to \infty}a(n)/n^b=0$ if $b >1/\ga$. 
We point out also that \eqref{eq:DoneyB}
has been first established for $\ga \in (1/2,1)$
in \cite{cf:GarsiaLamperti}.

\smallskip

\begin{proof}[Proof of Proposition \ref{th:homog}]
    We put for simplicity of notation $v(N):=N^{\ga}/L(N)$.
  Decomposing $Z_{N }$ with respect to the cardinality of $\tau\cap
  \{1,\ldots,N\}$,
\begin{equation}
\label{eq:lastline}
\begin{split}
  Z_{N}(-r(N)/v(N))\, &=\,\sum_{n=1}^{N}
  \bP\left(|\tau\cap\{1,\ldots,N\}|=n,N\in
  \tau\right)e^{-n\,r(N)/v(N)}\\
  &=\, \sum_{n=1}^{N}\bP(\tau_n=N)e^{-n\,r(N)/v(N)}\\
&=\, 
  \sum_{n=1}^{\frac{v(N)}{\sqrt{r(N)}}}\bP(\tau_n=N)e^{-n\,\frac{r(N)}{v(N)}}+
  \sum_{n=\frac{v(N)}{\sqrt{r(N)}}+1}^N\bP(\tau_n=N)e^{-n\,\frac{r(N)}{v(N)}}.
  \end{split}
\end{equation}
Observe now that one can rewrite  the first term in  the last line of \eqref{eq:lastline}  as
\begin{equation}
  (1+o(1))K(N) \sum_{n=1}^{v(N)/\sqrt{r(N)}}n \,e^{-n\,r(N)/v(N)},
\end{equation}
and $o(1)$ is a quantity which vanishes for $N\to\infty$
(this follows from Proposition
\ref{th:Don}, which applies uniformly over all terms of the sum in view
of $\lim_N r(N)=\infty$).  Thanks to
condition \eqref{eq:condr}, one can estimate this sum by an integral:
\begin{align*}
  \sum_{n=1}^{v(N)/\sqrt{r(N)}}n\,
  e^{-n\,r(N)/v(N)}=\frac{v(N)^2}{r(N)^2} (1+o(1))\int_{0}^\infty\dd x\,
  x\,e^{-x}= \frac{v(N)^2}{r(N)^2}(1+o(1)).
\end{align*}
As for the second sum in \eqref{eq:lastline}, observing that
$\sum_{n\in\N}\bP(\tau_n=N)=\bP(N\in\tau)$, we can bound it above by
\begin{equation}
\bP(N\in\tau)e^{-\sqrt{r(N)}}.
\end{equation}
In view of  \eqref{eq:DoneyB}, the last term is
negligible with respect to $N^{\ga-1}/(L(N)\,r(N)^2)$
and our result is proved.
\end{proof}

\noindent
{\it Proof of Lemma \ref{th:lemmaU2}.}
Recalling the notation
\eqref{eq:Zpure},  point (2) of Theorem \ref{th:pure} (see in 
particular the definition of $\hat L(\cdot)$) and  
\eqref{eq:DoneyB},
we see that the result we are looking for follows if we can show that
for every $c>0$ there exists $C_{23}=C_{23}(c)>0$ such that
\begin{equation}
\bE \left[e^{c|\tau\cap\{1,\ldots,N\}|L(N)/N^\ga}\Big \vert \, N\in\tau
\right]\, \le C_{23},
\end{equation}
uniformly in $N$.
Let us assume that $N/4\in \N$; by Cauchy-Schwarz inequality the
result follows if we can show that
\begin{equation}
\label{eq:withc}
\bE \left[e^{2c|\tau\cap\{1,\ldots,N/2\}|L(N)/N^\ga}\Big \vert \, N\in\tau
\right]\, \le C_{24}.
\end{equation}
Let us define $X_N :=\max\{n=0,1, \ldots, N/2:\, n \in \tau\}$
(last renewal epoch up to $N/2$).
By the renewal property  we have 
\begin{multline}
\bE \left[e^{2c|\tau\cap\{1,\ldots,N/2\}|L(N)/N^\ga}\Big \vert \, N\in\tau
\right]\\ 
\, =\, \sum_{n=0}^{N/2} 
\bE \left[e^{2c|\tau\cap\{1,\ldots,N/2\}|L(N)/N^\ga}\Big \vert \,X_N=n
\right] \bP\left( X_N=n \big \vert \, N \in \tau\right).
\end{multline}
If we can show that for every $n =0, 1, \ldots, N/2$
\begin{equation}
\bP\left( X_N=n \big \vert \, N \in \tau\right) \le C_{25} 
\bP\left( X_N=n\right),
\end{equation}
then we are reduced to proving \eqref{eq:withc} with
$\bE[\cdot \vert N \in \tau]$ replaced by $\bE[\cdot]$.

Let us then observe that
\begin{equation}
\label{eq:justbefore}
\begin{split}
\bP\left( X_N=n ,  \, N \in \tau\right)\, &=\,
\bP(n \in \tau) \bP\left( \tau_1> (N/2) -n, \, N-n \in \tau\right)
\\ 
&=\,\bP(n \in \tau)
\sum_{j= (N/2)-n+1}^{N-n} \bP( \tau_1=j)  \bP \left( N-n-j \in \tau\right).
\end{split}
\end{equation}
We are done if we can show that
\begin{equation}
\label{eq:tobesplit}
\sum_{j= (N/2)-n+1}^{N-n} \bP( \tau_1=j)  \bP \left( N-n-j \in \tau\right)
\,  \le \, 
C_{26}
\bP \left( N \in \tau\right)
\sum_{j= (N/2)-n+1}^{\infty} \bP( \tau_1=j),
\end{equation}
because the mass renewal function $\bP(N\in\tau)$ cancels when we consider the
conditioned probability and, recovering $\bP(n \in \tau)$ from
\eqref{eq:justbefore} we rebuild $\bP(X_N=n)$.  We split the sum in
the left-hand side of \eqref{eq:tobesplit} in two terms.  By using
\eqref{eq:DoneyB} (but just as upper bound) and the fact that the
inter-arrival distribution is regularly varying we obtain
\begin{multline}
\label{eq:term2}
\sum_{j= (3N/4)-n+1}^{N-n} \bP( \tau_1=j)  \bP \left( N-n-j \in \tau\right)
\\
\,  \le \, C_{27} \frac{L(N)}{N ^{1+\ga }} 
\sum_{j= (3N/4)-n+1}^{N-n} \frac1{(N-n-j+1)^{1-\ga }L(N-n-j+1)}
\\ =\, 
C_{27} \frac{L(N)}{N ^{1+\ga }} 
\sum_{j= 1}^{N/4}
\frac1{j^{1-\ga }L(j)}\, \le \, \frac{C_{28}}N. 
\end{multline}
Since the right-hand side of \eqref{eq:tobesplit} is bounded below by
$1/N$ times a suitable constant (of course if $n$ is 
%%F 
close to $N/2$ this quantity is sensibly larger) this first term of the
splitting is under control.  Now the other term: since the renewal
function is regularly varying
\begin{equation}
\label{eq:term1}
\sum_{j= (N/2)-n+1}^{(3N/4)-n} \bP( \tau_1=j) \bP \left( N-n-j \in
\tau\right) \, \le \, C_{29} \bP \left( N \in \tau\right)\sum_{j=
(N/2)-n+1}^{(3N/4)-n} \bP( \tau_1=j),
\end{equation}
that gives what we wanted.

It remains to show that \eqref{eq:withc} holds without conditioning.
For this we use the asymptotic estimate $-\log \bE[\exp(- \gl \tau_1)]
\stackrel{\gl \searrow 0}\sim c_\ga \gl ^\ga L(1/\gl)$, with $c_\ga =
\int_0^\infty r^{-1-\ga} (1-\exp(-r)) \dd r=\Gamma(1-\ga)/\ga$, and
the Markov inequality to get that if $x>0$
\begin{equation}
\label{eq:withoutc}
\bP\left( 
|\tau\cap\{1,\ldots,N\}|L(N)/N^\ga >x
\right)
\, =\, 
\bP\left( \tau_n <N \right) \, \le \, 
\exp \left( - \frac 12 c_\ga \gl^\ga L(1/\gl) n  + \gl N\right),
\end{equation}
with $n$  the integer part of $x N^\ga /L(N)$ and 
$\gl \in (0, \gl_0)$ for some $\gl_0>0$.  
If one chooses  $\gl = y/N$, $y$ a positive number, then for 
$x\ge 1$ and $N$ sufficiently large (depending on $\gl_0$ and $y$)
we have that the quantity at the exponent in the right-most term in
\eqref{eq:withoutc} is bounded above by
$ -(c_\ga/3) y^\ga x  +y$. 
The proof is then complete if we select $y$ such that
$(c_\ga/3) y^\ga> 2c$ ($c $ appears in \eqref{eq:withc}) since
if $X$ is a non-negative random variable and $q$ is a real number
$\bE[\exp(qX)] =1+q\int_0^\infty e^{qx} \bP(X>x) \dd x$.

 \subsection{Some basic facts about slowly varying functions}

\label{sec:Asv}

We recall here some of the elementary properties 
of slowly varying functions which
we repeatedly use, and we refer to \cite{cf:RegVar} for a complete
treatment of slow variation.

The first two well-known facts are that, if $U(\cdot)$ is slowly
varying at infinity,
\begin{equation}
\label{eq:sommasv}
  \sum_{n\ge N}\frac{U(n)}{n^m}\stackrel{N\to\infty}\sim U(N)\frac{N^{1-m}}
{m-1},
\end{equation}
if $m>1$ and 
\begin{equation}
  \label{eq:sommasv2}
  \sum_{n=1}^N \frac{U(n)}{n^m}\stackrel{N\to\infty}\sim U(N)\frac{N^{1-m}}{1-m},
\end{equation}
if $m<1$ ({\sl cf.} for instance \cite[Sec. A.4]{cf:GB}). 
 The second two facts are that
({\sl cf.} \cite[Th. 1.5.3]{cf:RegVar})
\begin{equation}
  \label{eq:minsv} 
\inf_{n\ge N} U(n) n^m\stackrel{N\to\infty}\sim U(N)\,N^m,
\end{equation}
if $m>0$, and
\begin{equation}
  \label{eq:maxsv}
\sup_{n\ge N} U(n) n^m\stackrel{N\to\infty}\sim U(N)\,N^m, 
\end{equation}
if $m<0$.

\qed

\section*{acknowledgments} 

G.G. and F.L.T.  acknowledge the support of the ANR grant {\sl POLINTBIO}.
B.D. and F.L.T. acknowledge the support of the ANR grant {\sl LHMSHE}.

\bigskip

%%G

\noindent
{\bf Note added in proof.}
After this work appeared in preprint form (arXiv:0712.2515 [math.PR]),
several new results have been proven. In \cite{cf:AZ} it has been shown
in particular that when $L(\cdot)$ is trivial, then $\gep$ in
Theorem~\ref{th:a121} can be chosen equal to zero, with $a(0)>0$. The
case $\ga=1$ is also treated in \cite{cf:AZ}.  The fractional moment
method we have developed here may be adapted to deal with the $\ga=1$
case too: this has been done in \cite{cf:BS}, where a related model is
treated. Finally, the controversy concerning the case $\alpha=1/2$
and $L(\cdot)$ asymptotically constant has been solved in \cite{cf:GLT08},
where it was shown that $h_c(\gb)>h_c^{ann}(\gb)$ for every $\beta>0$.


\begin{thebibliography}{99}


\bibitem{cf:AM}
M. Aizenman and S. Molchanov, \emph{Localization at large disorder and at extreme energies: an elementary
derivation}, Commun. Math. Phys. {\bf 157} (1993), 245--278.

\bibitem{cf:ASFH}
M. Aizenman, J. H. Schenker, R. M.  Friedrich and
D. Hundertmark, \emph{Finite-volume fractional-moment criteria for Anderson localization}, Commun. Math.
 Phys. {\bf 224} (2001),  219-253.

\bibitem{cf:Ken} K.~S.~Alexander, \emph{The effect of disorder on
    polymer depinning transitions}, Commun. Math. Phys. {\bf 279} (2008),  117-146.
      
  \bibitem{cf:AS} K. S. Alexander and V. Sidoravicius, \emph{Pinning of polymers and interfaces by random potentials}, Ann. Appl. Probab. {\bf 16} (2006), 636-669.

%\bibitem{cf:Ken_eurandom} K. S. Alexander, talk given at the workshop {\sl Random Polymers}, EURANDOM, June 2007.

 \bibitem{cf:AZ} K.~S.~Alexander and N.~Zygouras,
\emph{Quenched and annealed critical points in polymer pinning models},
arXiv:0805.1708 [math.PR]


\bibitem{cf:RegVar} N. H. Bingham, C. M. Goldie and J. L. Teugels,
  \emph{Regular variation}, Cambridge University Press, Cambridge,
  1987.

\bibitem {cf:BS} 
M.~Birkner and  R.~Sun,
\emph{Annealed vs quenched critical points for a random walk pinning model}, 
	arXiv:0807.2752 [math.PR]


\bibitem{cf:BCT} E. Bolthausen, F.  Caravenna and  B. de Tili\`ere, \emph{
The quenched critical point of a diluted disordered polymer model},
Stochastic Process. Appl. (to appear), 
arXiv:0711.0141 [math.PR]


\bibitem{cf:BPP} E. Buffet, A. Patrick and J. V.  Pul\'e,
\emph{Directed polymers on trees: a martingale approach}, J. Phys. A
Math. Gen. {\bf 26} (1993), 1823--1834.

\bibitem{cf:CCFS} J. T. Chayes, L. Chayes, D. S. Fisher and 
T. Spencer, \emph{Finite-size scaling and correlation lengths for
disordered systems}, Phys. Rev. Lett. {\bf 57} (1986), 2999--3002.

\bibitem{cf:coluzzi} B. Coluzzi and E. Yeramian, \emph{Numerical
evidence for relevance of disorder in a Poland-Scheraga DNA
denaturation model with self-avoidance: Scaling behavior of average
quantities}, Eur. Phys. Journal B {\bf 56} (2007), 349-365.

\bibitem{cf:DHV} B. Derrida, V. Hakim and J. Vannimenus,
  \emph{Effect of disorder on two-dimensional wetting}, J. Statist.
  Phys. {\bf 66} (1992), 1189--1213.

\bibitem{cf:Doney} R. A. Doney, \emph{One-sided local large deviation and 
renewal theorems in the case of infinite mean}, Probab. Theory Rel. Fields
{\bf 107} (1997), 451-465.

\bibitem{cf:vD} H.
von Dreifus, \emph{Bounds on the critical exponents of disordered ferromagnetic 
models}, Ann. Inst. H. Poincar\'e Phys. Th\'eor. {\bf 55} (1991), 657-669.


\bibitem{cf:ED}
M. R. Evans and  B. Derrida, \emph{Improved bounds for the transition temperature of directed polymers in a
finite-dimensional random medium}, J. Statist. Phys. {\bf 69} (1992), 427--437.

\bibitem{cf:Fisher}
M.~E.~Fisher, \emph{ Walks, walls, wetting, and melting},
J. Statist. Phys. {\bf 34} (1984), 667--729.

\bibitem{cf:FLNO} G. Forgacs, J. M. Luck, Th. M. Nieuwenhuizen and
H. Orland, \emph{Wetting of a disordered substrate: exact critical
behavior in two dimensions}, Phys. Rev. Lett. {\bf 57} (1986),
2184--2187.

\bibitem{cf:GN} D. M. Gangardt and S. K. Nechaev, \emph{Wetting
	transition on a one-dimensional disorder}, 
	 J. Statist. Phys. {\bf 130} (2008),  483-502. 

\bibitem{cf:GarsiaLamperti} A. Garsia and J. Lamperti, \emph{
A discrete renewal theorem with infinite mean},
Comment. Math. Helv. {\bf 37} (1963), 221--234.

\bibitem{cf:GB} G. Giacomin, \emph{ Random Polymer Models},
Imperial College Press, World Scientific (2007).

\bibitem{cf:GLT} G. Giacomin, H. Lacoin and F. L. Toninelli, \emph{
    Hierarchical pinning models, quadratic maps and quenched
    disorder}, arXiv:0711.4649 [math.PR]

\bibitem{cf:GLT08} G. Giacomin, H. Lacoin and F. L. Toninelli, \emph{
Marginal relevance of disorder for pinning models},  arXiv:0811.0723 [math-ph]

\bibitem{cf:GTdeloc} G.~Giacomin and F.~L.~Toninelli, \emph{Estimates on
    path delocalization for copolymers at selective interfaces},
  Probab. Theor. Rel. Fields {\bf 133} (2005), 464--482.

\bibitem{cf:GT_cmp} G.~Giacomin and  F.~L.~Toninelli, \emph{
Smoothing effect of quenched disorder on polymer depinning transitions},
Commun. Math. Phys. {\bf 266} (2006), 1--16.

\bibitem{cf:GT_irrel} G. Giacomin and F.~L.~Toninelli, \emph{ On the
    irrelevant disorder regime of pinning models}, preprint (2007).
    arXiv:0707.3340 [math.PR]

  \bibitem{cf:Harris} A.~B.~Harris, \emph{ Effect of Random Defects on
      the Critical Behaviour of Ising Models}, J. Phys. C {\bf 7}
    (1974), 1671--1692.
 
\bibitem{cf:dna} Y. Kafri, D. Mukamel and L. Peliti, \emph{ Why is the
DNA denaturation transition first order?}, Phys. Rev.
Lett. {\bf 85}
(2000), 4988--4991.
				   
%\bibitem{cf:MG} C. Monthus and T. Garel, {\sl On the multifractal statistics of the local order parameter at random critical points: application to wetting transitions with disorder}, Phys. Rev. E {\bf 76} (2007), 021114.
 
\bibitem{cf:T_cmp}  F.~L.~Toninelli,
\emph{A replica-coupling approach to disordered pinning models},
 Commun. Math. Phys. {\bf 280} (2008),  389-401.

\bibitem{cf:T_AAP} F. L. Toninelli, \emph{Disordered pinning models and
    copolymers: beyond annealed bounds},
    Ann. Appl. Probab. {\bf 18} (2008),  1569-1587.
    
\end{thebibliography}
\end{document}